\documentclass[11pt]{article}
\usepackage{amsfonts}
\usepackage{amsmath}
\usepackage{amssymb}
\usepackage{amsthm}
\usepackage{amstext}
\usepackage{mathtools}
\usepackage{latexsym}
\usepackage{physics}
\usepackage{xspace}
\usepackage{ifthen}
\usepackage{boxedminipage}
\usepackage{paralist}

\advance \topmargin by -0.7in
\advance \oddsidemargin by -0.8in

\newtheorem{theorem}{Theorem}
\newtheorem{lemma}{Lemma}

\newtheorem{corollary}{Corollary}

\newcommand{\RR}{{\mathbb R}}
\newcommand{\CC}{{\mathbb C}}

\newcommand{\E}{{\mathbb E}}
\newcommand{\Var}{\mathsf{Var}}
\newcommand{\Cov}{\mathsf{Cov}}

\def\b0{{\bf 0}}
\def\b1{{\bf 1}}

\title{Concentration of Lipschitz Functions \\ of Negatively Dependent Variables}

\author{Kevin Garbe \and Jan Vondr\'ak}

\date{\today}

\begin{document}

\maketitle

\abstract{We explore the question whether Lipschitz functions of random variables under various forms of negative correlation satisfy concentration bounds similar to McDiarmid's inequality for independent random variables. We prove such a concentration bound for random variables satisfying the condition of {\em negative regression}, correcting an earlier proof \cite{DubhashiRanjan98}.}

\section{Introduction}

We study the question whether functions of negatively dependent random variables satisfy concentration bounds, similar to functions of independent random variables.
Many tools are known for functions of independent random variables --- e.g., martingales, Talagrand's inequality, the Kim-Vu inequality, and the entropy method for self-bounding functions. It is also known that Chernoff-Hoeffding bounds for (linear functions of) independent random variables generalize to negatively dependent variables, assuming a relatively weak property known as negative cylinder dependence. A natural question therefore arises, whether more sophisticated concentration bounds also generalize to some form of negative dependence.
In this paper, we discuss the question of concentration bounds for {Lipschitz functions}\footnote{A function $f:\{0,1\}^n \rightarrow \RR$ is $c$-Lipschitz, if each variable can affect the value additively by at most $c$.}  on $\{0,1\}^n$, under certain forms of negative dependence.

\subsection{Prior work}

Negatively dependent random variables arise naturally in scenarios motivated by statistical physics as well as computer science.
Newman studied the asymptotic behavior of ensembles of random variables under a certain notion of negative dependence, and proved a central limit theorem in this setting \cite{Newman84}.
Panconesi and Srinivasan \cite{PancoSrini97} proved that random variables under the condition of {\em negative cylinder dependence} (see Section~\ref{sec:neg-dep} for definitions) satisfy Chernoff-Hoeffding concentration bounds, just like independent random variables. It was shown that negative cylinder dependence is exhibited by random variables arising in various randomized rounding scenarios which led to several applications \cite{PancoSrini97,GKPS06,AgeevS04,CVZ10}. We note that from the point of view of this paper, these concentration results are somewhat special as they only apply to {\em linear} functions of negatively dependent variables.

The notion of negative dependence was studied systematically by Pemantle \cite{Pemantle00} who proposed several alternative notions of negative dependence and investigated their relative merits and relationships. In particular, he posed the question whether there is a robust notion of negative dependence, closed under natural operations, and allowing one to replicate some of the theory enjoyed by independent random variables or positively dependent variables (such as the FKG inequality).

Meanwhile, Dubhashi and Ranjan \cite{DubhashiRanjan98} studied in depth the scenario of {\em balls and bins}, and proved that the respective random variables satisfy several notions of negative dependence. Among other results, they stated a concentration bound for any Lipschitz function of random variables satisfying the property of {\em negative regression}. This was a rare example of a concentration inequality for non-linear functions of negatively dependent variables that can be found in the literature --- unfortunately, it turns out that their proof was erroneous (see Appendix~\ref{app:error} for details).
Later, Farcomeni  \cite{Farcomeni08} studied concentration bounds under a weaker notion of negative dependence (which reduces to negative cylinder dependence in the case of $\{0,1\}$ random variables) and claimed a concentration inequality for Lipschitz functions of such variables. Unfortunately, his proof again turned out to be incorrect, as reported by \cite{PemPeres14}. In that work, Pemantle and Peres \cite{PemPeres14} proved a concentration bound for Lipschitz functions under {\em strong Rayleigh measures}, a strong notion of negative dependence which implies all the other notions discussed in this paper. Their proof relies on the theory of stable polynomials \cite{BBL09}.

Let us mention that special-purpose tail inequalities have been proved for submodular functions and matrix norms under a particular randomized rounding scheme on matroid polytopes \cite{CVZ09,HarveyOlver14}. While this is a natural setting where negatively dependent variables arise, it is not known how to generalize these concentration bounds to any general notion of negative dependence. 

\paragraph{Organization.}
In Section~\ref{sec:neg-dep}, we survey several notions of negative dependence and their relationships.
In Section~\ref{sec:neg-reg-result}, we provide a corrected proof of the concentration bound for Lipschitz functions of random variables under the assumption of negative regression. This is the main contribution of this paper.
In Appendix~\ref{app:error}, we explain where the proof of \cite{DubhashiRanjan98} fails and discuss a counterexample that motivated our proof.

\section{Notions of negative dependence}
\label{sec:neg-dep}

Let us survey here several notions of negative dependence and their known relationships. 
In this paper, we restrict attention to (correlated) Bernoulli random variables, i.e.~probability measures on $\{0,1\}^n$.

\paragraph{Pairwise Negative Correlation.}
This is the weakest notion of negative dependence considered here. 
Variables $X_1,\ldots,X_n$ are pairwise negatively correlated, if $$\Cov[X_i,X_j] = \E[X_i X_j] - \E[X_i] \E[X_j] \leq 0$$ for all $i,j \in [n]$.
Pairwise negative correlation allows one to use Chebyshev's inequality, but it is too weak to imply exponential tail bounds.
For example, $n$ pairwise independent variables defined by the $n \times n$ Hadamard matrix have probability $\Omega(1/n)$ of being all equal to $1$.

\paragraph{Negative Cylinder Dependence.}
$X_1,\ldots,X_n$ are negative cylinder dependent, if for every $S \subseteq [n]$,
$$ \E[ \prod_{i \in S} X_i] \leq \prod_{i \in S} \E[X_i] $$
and
$$ \E[ \prod_{i \in S} (1-X_i)] \leq \prod_{i \in S} \E[1-X_i].$$
Negative cylinder dependence is strictly stronger than pairwise negative correlation. It is known to imply exponential concentration bounds for linear functions of $X_1,\ldots,X_n$.

\paragraph{Negative Association.}
$X_1,\ldots,X_n$ are negatively associated if for any $I,J \subset [n], I \cap J = \emptyset$ and any pair of non-decreasing functions $f:\{0,1\}^I \rightarrow \RR$, $g:\{0,1\}^J \rightarrow \RR$, $$\E[f(X_I) g(X_J)] \leq \E[f(X_I)] \E[g(X_J)].$$
(Here and in the following, $X_S \in \{0,1\}^S$ denotes the $|S|$-tuple of random variables indexed by $S$.)

Negative association is strictly stronger than negative cylinder dependence. Whether negative association implies exponential concentration bounds for Lipschitz functions is an interesting question (posed by Elchanan Mossel \cite{PemPeres14} and still open as far as we know). 

\paragraph{Negative Regression.}
$X_1,\ldots,X_n$ satisfy negative regression, if for any $I,J \subset [n], I \cap J = \emptyset$, any non-decreasing function $f:\{0,1\}^I \rightarrow \RR$ and $a \leq b \in \{0,1\}^J$,
$$ \E[f(X_I) \mid X_J = a] \geq \E[f(X_I) \mid X_J = b].$$
Negative regression is the main subject of this paper; we prove a concentration bound for Lipschitz functions under negative regression here.
The relationship of negative association and negative regression is not completely understood. It is known that negative association does not imply negative regression \cite{DubhashiRanjan98} but the status of the opposite implication is unknown. Negative regression is preserved under conditioning of variables, while negative association is not. It is easy to see that both properties are implied by the following.

\paragraph{Conditional Negative Association.}
$X_1,\ldots,X_n$ are conditionally negatively associated, if for any $I \subset [n]$ and $a \in \{0,1\}^I$, $(X_j: j \in [n] \setminus I)$ conditioned on $X_I = a$ are negatively associated.

\

Pemantle \cite{Pemantle00} conjectured that conditional negative association is equivalent to negative regression. 
He also defined several related notions that entail the ``negative lattice condition" (log-submodularity) and ``external fields". We will not discuss these here.

\paragraph{Stochastic Covering.}
$X_1,\ldots,X_n$ satisfy the stochastic covering property, if for any $I \subset [n]$ and $a \geq a' \in \{0,1\}^I$ such that $\| a' - a \|_1 = 1$, there is a coupling $\nu$
of the distributions $\mu, \mu'$ of $(X_j: j \in [n] \setminus I)$ conditioned on $X_I = a$ or $X_I=a'$, respectively, such that $\nu(x,y)=0$ unless $x \leq y$ and $\|x-y\|_1 \leq 1$.

\

This property is discussed in \cite{PemPeres14}. It is stronger than negative regression (which is equivalent to a similar coupling condition, without the requirement that the coupling is supported on pairs of distance at most $1$; this follows from Strassen's theorem as we discuss further).

\paragraph{The Strong Rayleigh Property.}
$X_1,\ldots,X_n$ satisfy the strong Rayleigh property if the generating function $F(z_1,\ldots,z_n) = \E[\prod_{j=1}^{n} z_j^{X_j}]$ is a real stable polynomial, which means it has no root $(z_1,\ldots,z_n) \in \CC^n$ with all imaginary components strictly positive.

\medskip

The strong Rayleigh property is intimately tied to the theory of stable polynomials, and it was shown in \cite{BBL09} to imply all the other forms of negative dependence discussed above. Hence, it can be viewed as a very strong and robust notion of negative dependence --- it is closed under several natural operations, and exhibits a number of other desirable properties. 
Particularly relevant to this paper is the result of \cite{PemPeres14} that any Lipschitz function of random variables under a strong Rayleigh measure satisfies concentration bounds similar to independent random variables. 
Furthermore, \cite{PemPeres14} also proves a concentration bound for Lipschitz functions for homogeneous random variables ($\sum_{i=1}^{n} X_i = k$ for some constant $k$, with probability $1$) satisfying the stochastic covering property.

\section{Concentration under negative regression}
\label{sec:neg-reg-result}

We provide here what (we claim) is a correct proof of a previously claimed result \cite[Proposition 31]{DubhashiRanjan98}, namely a concentration bound for Lipschitz functions of random variables under the property of negative regression. In fact we improve their theorem in the sense that we do not require the assumption of monotonicity of $f$.
For simplicity, we focus on the case of $1$-Lipschitz functions.

\begin{theorem}
\label{thm:Lipschitz-R}
Let $X_1,\ldots,X_n$ be $\{0,1\}$ random variables satisfying the condition of negative regression.
Let $f:\{0,1\}^n \rightarrow \RR$ be a $1$-Lipschitz function, and $\mu = \E[f(X_1,\ldots,X_n)]$. Then for any $t > 0$,
$$ \Pr[f(X_1,\ldots,X_n) \geq \mu + t] \leq e^{-t^2/2n},$$
$$ \Pr[f(X_1,\ldots,X_n) \leq \mu - t] \leq e^{-t^2/2n}.$$
If $f$ is monotone, then the bound can be improved to $e^{-2 t^2 / n}$.
\end{theorem}

We remark that in the case of monotone $f$, our bound coincides with that of McDiarmid's inequality for independent random variables \cite{McDiarmid89},
in which case the constant in the exponent is known to be tight. In the non-monotone case, it is known that the same constant cannot be achieved \cite{PemPeres14}, but it might be possible to prove a bound of $e^{-t^2/n}$; i.e., our constant could be off by a factor of $2$. 

Negative regression is weaker than stochastic covering, or the strong Rayleigh property, and hence qualitatively our bound subsumes that of \cite{PemPeres14}. 
In terms of applications, it is fair to say that most of the known examples satisfying negative regression in fact satisfy the strong Rayleigh property and are hence covered by \cite{PemPeres14}.
Quantitatively speaking, the constants in the exponent proved by Pemantle and Peres \cite{PemPeres14} are somewhat worse than ours, but their bounds are functions of $\mu = \E[\sum X_i]$ rather than $n$, so are not directly comparable to ours.

In terms of techniques, our proof is more elementary than that of \cite{PemPeres14}, which relies on the (beautiful) theory of stable polynomials. 
We follow the classical martingale paradigm, with one new twist --- an adaptive ordering of the variables.

\subsection{Adaptive martingale analysis}

As we show in Appendix~\ref{app:error}, it is not possible to replicate the martingale analysis of McDiarmid's inequality under negative regression, if we work with a fixed ordering of variables $X_1,\ldots,X_n$. The problem is that a particular variable $X_1$ might have a large influence on the distribution of the remaining variables, and hence affect significantly the conditional expectation of $f$ once the value of $X_1$ is revealed. 

Our way around this issue is that we can choose variables {\em adaptively}, in order to construct a martingale with bounded differences. Our goal is to choose a variable that does not affect the remaining variables too heavily. Due to negative regression, we know that conditioning on $X_i=1$ can only affect the remaining variables negatively. Hence the only thing we have to worry about is that this negative effect is too large. However, the following lemma shows that there always exists a variable whose negative influence on the remaining variables is not too large.

\begin{lemma}
\label{lem:pick}
Let $X_1, \ldots, X_n$ be $\{0,1\}$ random variables.
Fix $K \subset [n]$ and $a_K \in \{0,1\}^K$ such that $Pr[X_K=a_K] > 0$. Then there exists $i \in [n] \setminus K$ such that  
either $X_i$ conditioned on $X_K=a_K$ is deterministic, or
$$ \sum_{\ell \in [n] \setminus (K \cup \{i\})} \left( \E[ X_\ell \mid X_K = a_K, X_i = 0] - \E[ X_\ell \mid X_K = a_K, X_i = 1] \right) \leq 1.$$
\end{lemma}

\begin{proof}
Let us denote $L = [n] \setminus K$. We consider the following quantity:
\begin{align*}
\Var\big[\sum_{i \in L} X_i \mid X_K=a_K\big] 
 = \sum_{i \in L} \left( \Var[X_i \mid X_K=a_K] + \sum_{j \in L \setminus \{i\}} \Cov[X_i, X_j \mid X_K=a_K] \right).
\end{align*}
Variance is always nonnegative, so at least one term of the summation over $i \in L$ must be non-negative:
\begin{equation}
\label{eq:1}
\Var[X_i \mid X_K=a_K] + \sum_{j \in L \setminus \{i\}} \Cov[X_i, X_j \mid X_K=a_K] \geq 0.
\end{equation}
Let us denote $\pi_i = \Pr[X_i=1 \mid X_K=a_K]$. We can write
$$ \Var[X_i \mid X_K=a_K] = \E[X_i^2 \mid X_K=a_K] - (\E[X_i \mid X_K=a_k])^2 = \pi_i (1-\pi_i) $$
and
\begin{eqnarray*}
\Cov[X_i,X_j \mid X_K=a_K] & = & \E[X_i X_j \mid X_K=a_K] - \E[X_i \mid X_K=a_K] \, \E[X_j \mid X_K=a_K] \\
 & = & \pi_i ( \E[X_j \mid X_K=a_K, X_i=1] - \E[X_j \mid X_K=a_K] ).
\end{eqnarray*}
Since $\E[X_j \mid X_K=a_K] = \pi_i  \E[X_j \mid X_K=a_K, X_i=1] + (1-\pi_i)  \E[X_j \mid X_K=a_K, X_i=0]$, we can rewrite this as
\begin{eqnarray*}
\Cov[X_i,X_j \mid X_K=a_K] = \pi_i (1-\pi_i) ( \E[X_j \mid X_K=a_K, X_i=1] -  \E[X_j \mid X_K=a_K, X_i=0]).
\end{eqnarray*}
If $\pi_i=0$ or $\pi_i=1$, then $X_i$ conditioned on $X_K=a_K$ is deterministic, and we are done. Otherwise, substitute the expressions for $\Var[X_i \mid X_K=a_K]$ and $\Cov[X_i,X_j \mid X_K=a_K]$ in (\ref{eq:1}), divide by $\pi_i (1-\pi_i)$, and conclude that
$$ \sum_{\ell \in L \setminus \{i\}} \left( \E[ X_\ell \mid X_K = a_K, X_i = 0] - \E[ X_\ell \mid X_K = a_K, X_i = 1] \right) \leq 1.$$ 
\end{proof}

Next, we show that conditioning on $X_i$ indeed cannot affect the conditional expectation of $f$ too much.
To prove this, we need one more tool, which is Strassen's monotone coupling theorem (easily proved from the max-flow min-cut theorem).

\begin{theorem}[Strassen's Theorem]
Consider two probability measures $\mu^{(0)}$ and $\mu^{(1)}$ on $\{0,1\}^L$.
Suppose that for every down-closed $M \subset \{0,1\}^L$, $\mu^{(1)}(M) \geq \mu^{(0)}(M)$.
Then there exists a coupling $\nu:\{0,1\}^L \times \{0,1\}^L \rightarrow [0,1]$ such that:
$$\mu^{(1)}(x) = \sum_y \nu(x, y),$$
$$\mu^{(0)}(y) = \sum_x \nu(x, y),$$
and $\nu(x, y) = 0$ unless $x \leq y$ coordinate-wise.
\end{theorem}

I.e., if $\mu^{(1)}$ dominates $\mu^{(0)}$ on every down-closed event, then it is possible to transform $\mu^{(0)}$ into $\mu^{(1)}$ in such a way that we only transfer probability mass {\em downwards} in $\{0,1\}^L$.
From the condition of negative regression and Strassen's theorem, we get immediately the following.

\begin{corollary}
\label{cor:coupling}
For random variables $X_1,\ldots,X_n$ satisfying negative regression, $K \subset [n]$, $a_K \in \{0,1\}^K$, and $i \in [n] \setminus K$ such that
$\Pr[X_K=a_k, X_i=0] > 0$, $\Pr[X_K=a_k, X_i=1] > 0$,  there exists a coupling $\nu:\{0,1\}^L \times \{0,1\}^L \rightarrow [0,1]$
for $L = [n] \setminus (K \cup \{i\})$ such that
$$ \mu^{(1)}(a_L) = \Pr[X_L=a_L \mid X_K=a_K, X_i=1] = \sum_y \nu(a_L,y),$$
$$ \mu^{(0)}(b_L) = \Pr[X_L=b_L \mid X_K=a_K, X_i=0] = \sum_x \nu(x,b_L),$$
and $\nu(x,y) = 0$ unless $x \leq y$.
\end{corollary}

Now we are ready to construct a martingale with bounded differences. We formalize this as follows.

\begin{lemma}
\label{lem:martingale}
Given $\{0,1\}$ random variables $X_1,\ldots,X_n$ satisfying negative regression, and a $1$-Lipschitz function $f:\{0,1\}^n \rightarrow \RR$,
there exists an adaptive ordering (random permutation) $\pi(1), \pi(2), \ldots, \pi(n)$ such that:
\begin{compactitem}
\item $\pi(1)$ is deterministic.
\item For each $1 \leq k<n$, $\pi(k+1)$ is determined by $\pi([k]) = \{\pi(1),\ldots,\pi(k)\} = K$ and $X_K = a_K$.
\item If we denote $Y_k = \E[f(X) \mid \pi([k]),X_{\pi([k])}]$, then $Y_0,Y_1,\ldots,Y_n$ is a martingale.
\item Conditioned on $\pi([k])=K, X_K=a_K$, there are $\alpha<\beta, \beta-\alpha \leq 2$, such that
$Y_{k+1} - Y_k \in [\alpha,\beta]$ with prob.~$1$.
\item For monotone $f$, $Y_{k+1} - Y_k \in [\alpha,\beta]$ where $\beta-\alpha \leq 1$.
\end{compactitem}
\end{lemma}

\begin{proof}
Let us fix $\pi([k]) = K$ and $X_K = a_K$.
Given this conditioning, there exists $i \in [n] \setminus K$ as provided by Lemma~\ref{lem:pick}.
Let us define $\pi(k+1)$ to be the minimum index $i$ satisfying the conclusion of Lemma~\ref{lem:pick}.

Now let us consider $Y_k = \E[f(X) \mid \pi([k])=K,X_K=a_K]$. 
Under this conditioning, $\pi(k+1)$ is deterministic and $X_{\pi(k+1)}$ could take two possible values ($0$ or $1$) which determines
the value of $Y_{k+1}$.  By construction, $Y_k = \E_{X_{\pi(k+1)}}[Y_{k+1} \mid \pi([k])=K, X_K=a_K]$.
By averaging over all possible choices of $K$ and $a_K$ consistent with the same value of $Y_k$, we also get $Y_k = \E[Y_{k+1} \mid Y_k]$; i.e.,
the sequence $Y_0,Y_1,\ldots,Y_n$ forms a martingale.

Our goal now is to analyze how much $Y_{k+1}$ can deviate from  $Y_k$. In the following we fix $\pi([k])=K$ and $X_K = a_K$.
We have $Y_k = \E[f(X) \mid \pi([k])=K, X_K=a_K]$. Recall that this conditioning also determines the choice of $\pi(k+1)$ (but of course not the value of $X_{\pi(k+1)}$).
If $X_{\pi(k+1)}$ attains only one possible value under this conditioning, then $Y_{k+1} = Y_k$ and we are done --- hence, let us assume
that both values of $X_{\pi(k+1)}$ occur with positive probability.
For $c=0,1$, let us denote 
$$Y_{k+1}^{(c)} = \E[f(X) \mid \pi(1),\ldots,\pi(k), \pi(k+1), X_{\pi(1)}, \ldots,X_{\pi(k)}; X_{\pi(k+1)}=c].$$
We claim that $|Y_{k+1}^{(1)} - Y_{k+1}^{(0)}| \leq 2$, and in the case of monotone $f$,  $|Y_{k+1}^{(1)} - Y_{k+1}^{(0)}| \leq 1$; showing this will complete the proof.

Let $L = [n] \setminus \pi([k+1]) = [n] \setminus (K \cup \{ \pi(k+1) \})$.
Denote by $\mu^{(c)}$ the probability distribution of $X_L \in \{0,1\}^L$ conditioned on $\pi([k])=K$, $X_K=a_K$, and $X_{\pi(k+1)} = c \in \{0,1\}$.
By Corollary~\ref{cor:coupling}, there is a monotone coupling $\nu(x,y)$ between $\mu^{(1)}$ and $\mu^{(0)}$. We can write
\begin{eqnarray*} 
Y^{(1)}_{k+1} & = & \E[f(X) \mid \pi([k])=K, X_K=a_K, X_{\pi(k+1)}=1] \\
& = & \sum_{a_L \in \{0,1\}^L} \mu^{(1)}(a_L) \, f(X_K=a_K, X_{\pi(k+1)}=1, X_L=a_L) \\
& = & \sum_{a_L,b_L \in \{0,1\}^L} \nu(a_L,b_L) \, f(X_K=a_K, X_{\pi(k+1)}=1, X_L=a_L).
\end{eqnarray*}
Similarly,
\begin{eqnarray*} 
Y^{(0)}_{k+1} & = & \E[f(X) \mid \pi([k])=K, X_K=a_K, X_{\pi(k+1)}=0] \\
& = & \sum_{b_L \in \{0,1\}^L} \mu^{(0)}(b_L) \, f(X_K=a_K, X_{\pi(k+1)}=0, X_L=b_L) \\
& = & \sum_{a_L, b_L \in \{0,1\}^L} \nu(a_L,b_L) \, f(X_K=a_K, X_{\pi(k+1)}=0, X_L=b_L).
\end{eqnarray*}
Now we can compare the values of $f$ in the two expressions. First we modify the coordinate $X_{\pi(k+1)}$, and then the remaining coordinates $X_L$.
Since $f$ is $1$-Lipschitz, we have
$$\abs{f(X_K=a_K, X_{\pi(k+1)}=1, X_L=a_L) - f(X_K=a_K, X_{\pi(k+1)}=0, X_L=a_L)} \leq 1.$$
Next, since  $a_L \leq b_L$ in all the terms with $\nu(a_L,b_L) > 0$, we have
$$\abs{f(X_K=a_k, X_{\pi(k+1)}=0, X_K=a_L) - f(X_K=a_k, X_{\pi(k+1)}=0, X_K=b_L)} \leq \| b_L \|_1 - \| a_L \|_1.$$
By the triangle inequality,
$$\abs{f(X_K=a_k, X_{\pi(k+1)}=1, X_K=a_L) - f(X_K=a_k, X_{\pi(k+1)}=0, X_K=b_L)} \leq 1 + \| b_L \|_1 - \| a_L \|_1. $$
Hence
\begin{eqnarray*}
 \abs{Y^{(1)}_{k+1} - Y^{(0)}_{k+1}} & \leq & \sum_{a_L,b_L \in \{0,1\}^L} \nu(a_L,b_L) (1 + \|b_L\|_1 - \|a_L\|_1) \\
  & = & 1 + \sum_{b_L \in \{0,1\}^L} \mu^{(0)}(b_L) \, \| b_L \|_1 - \sum_{a_L \in \{0,1\}^L} \mu^{(1)}(a_L)\,  \| a_L \|_1 \\
  & = & 1 + \sum_{\ell \in L} \E[X_\ell \mid X_K=a_K, X_{\pi(k+1)}=0] - \sum_{\ell \in L} \E[X_\ell \mid X_K=a_K, X_{\pi(k+1)}=1]
\end{eqnarray*}
by the properties of $\mu^{(0)}, \mu^{(1)}$ and $\nu$. 
Finally, we recall that $\pi(k+1)$ was chosen so as to satisfy the conclusion of  Lemma~\ref{lem:pick}: 
$$\sum_{\ell \in L} \E[X_\ell \mid X_K=a_K, X_{\pi(k+1)}=0] - \sum_{\ell \in L} \E[X_\ell \mid X_K=a_K, X_{\pi(k+1)}=1] \leq 1.$$
This concludes the proof that $|Y^{(1)}_{k+1} - Y^{(0)}_{k+1}| \leq 2$.

In the case of monotone $f$, we observe that we can improve some of the inequalities: we get
$$0 \leq f(X_K=a_K, X_{\pi(k+1)}=1, X_L=a_L) - f(X_K=a_K, X_{\pi(k+1)}=0, X_L=a_L) \leq 1 $$
and 
$$ 0 \geq f(X_K=a_k, X_{\pi(k+1)}=0, X_K=a_L) - f(X_K=a_k, X_{\pi(k+1)}=0, X_K=b_L)| \geq \| a_L \|_1 - \| b_L \|_1 $$
whenever $a_L \leq b_L$. This implies that 
$$ 1 \geq f(X_K=a_k, X_{\pi(k+1)}=1 X_K=a_L) - f(X_K=a_k, X_{\pi(k+1)}=0, X_K=b_L) \geq \| a_L \|_1 - \| b_L \|_1 $$
and by the same computations as above, we conclude that $1 \geq Y^{(1)}_{k+1} - Y^{(0)}_{k+1} \geq -1$.
\end{proof}

Theorem~\ref{thm:Lipschitz-R} now follows by standard exponential moment analysis; see for example \cite{McDiarmid89}.
For completeness, we summarize the rest of the analysis as follows.

\begin{proof}[Proof of Theorem~\ref{thm:Lipschitz-R}]
For a parameter $\lambda \in \RR$, we estimate the exponential moment $\E[e^{\lambda f(X_1,\ldots,X_n)}] = \E[e^{\lambda Y_n}]$.
Inductively, we prove that $\E[e^{\lambda (Y_k - Y_0)}] \leq e^{k \lambda^2 / 2}$ (for monotone $f$, the bound improves to $e^{k \lambda^2 / 8}$).
The inductive step is that
$$ \E[e^{\lambda (Y_{k+1} - Y_k)} \mid \pi([k])=K, X_K=a_K] \leq e^{\lambda^2 / 2} $$
which is true because under this conditioning, $Y_{k+1} - Y_k$ is a random variable of expectation $0$, confined to an interval of length $2$ (see Lemma~\ref{lem:martingale}).
It is known that the exponential moment of such a variable is upper-bounded by $e^{\lambda^2 / 2}$. 
By averaging over all choices of $K$ and $a_K$ that yield the same value of $Y_k$, we also obtain $ \E[e^{\lambda (Y_{k+1} - Y_k)} \mid Y_k] \leq e^{\lambda^2 / 2}.$
In the case of $f$ monotone, $Y_{k+1} - Y_k$ is confined to an interval of length $1$ and the exponential moment is upper-bounded by $e^{\lambda^2 / 8}$.

From here, by induction, we obtain
$$ \E[e^{\lambda (Y_{k+1} - Y_0)}] = \E_{Y_k}[ \E_{Y_{k+1}}[e^{\lambda (Y_{k+1} - Y_k)} \mid Y_k] \ e^{\lambda (Y_k-Y_0)}] \leq e^{\lambda^2 / 2} \ \E_{Y_k}[e^{\lambda(Y_k-Y_0)}] \leq e^{(k+1) \lambda^2 / 2}.$$
Therefore, $\E[e^{\lambda (Y_n - Y_0)}] \leq e^{n \lambda^2 / 2}$. (For $f$ monotone, the bound improves to $e^{n \lambda^2 / 8}$.)

Recall that $Y_0 = \E[f(X_1,\ldots,X_n)] = \mu$.
By Markov's inequality applied to the exponential moment, we get
$$ \Pr[f(X_1,\ldots,X_n) > \mu+t] = \Pr[e^{\lambda (Y_n - Y_0)} \geq e^{\lambda t}] \leq \frac{\E[e^{\lambda (Y_n-Y_0)}]}{e^{\lambda t}} = e^{n \lambda^2 / 2 - \lambda t}.$$
The choice of $\lambda = t/n$ gives the upper-tail bound in Theorem~\ref{thm:Lipschitz-R}; the other bounds follow similarly.
\end{proof}

\section{Discussion and conclusion}

Let us discuss briefly the notion of negative regression and how it relates to other notions of negative dependence.
It is fair to say that most known examples that satisfy negative association or negative regression actually satisfy the strong Rayleigh property as well.
For example, random variables arising in the context of random spanning trees, determinantal point processes and exclusion dynamics processes are in this category.
However, there are ensembles of random variables satisfying negative regression and not stronger notions of negative dependence. We want to mention a few examples here.

\paragraph{Random variables conditioned on their sum.}
It is known that if $X_1,\ldots,X_n$ are independent, then the probability measure of $(X_1,\ldots,X_n)$ conditioned on $\sum_{i=1}^{n} X_i = k$ is strongly Rayleigh, for any fixed $k$. The probability measure conditioned on $\sum_{i=1}^{n} X_i \in \{ k, k+1 \}$ is still strongly Rayleigh. Conditioning on $\sum_{i=1}^{n} X_i \in [a,b]$ does not preserve the strong Rayleigh condition in general \cite{BBL09}. However, such ensembles still satisfy negative regression \cite{Pemantle00}. More generally, ensembles produced from independent random variables by taking products, ``external fields" and ``rank rescaling" satisfy negative regression; we refer the reader to \cite{Pemantle00} for a precise statement and proof.

\paragraph{Variables of large influence.}
As we remarked, \cite{PemPeres14} also handles the case of random variables satisfying the stochastic covering property and homogeneity, i.e. the condition $\sum_{i=1}^{n} X_i = k$. The case of conditioning on $\sum_{i=1}^{n} X_i \in [a,b]$ is not covered by their theorem, although we believe that their method would still apply. However, what seems significantly beyond the scope of stochastic covering is the case of random variables where one variable can have a large effect on the remaining variables. (Under stochastic covering, conditioning on one variable can change the expected sum of the remaining variables by at most $1$.)
An example of such a measure is our counterexample in Appendix~\ref{app:error}; this counterexample illustrates what the issue was in the previously claimed proof, and also this is the main hurdle that our proof had to overcome.

\

The main question that this paper leaves open is whether similar concentration bounds still hold for Lipschitz functions of negatively associated variables.


\appendix

\section{The failure of a fixed ordering}
\label{app:error}

Here we review briefly the argument presented in \cite{DubhashiRanjan98} and why it is flawed. In Proposition 31, \cite{DubhashiRanjan98} states Theorem~\ref{thm:Lipschitz-R} under the assumption of $f$ being monotone and Lipschitz with constant $c_i$ in variable $X_i$; here let us assume $c_i=1$.
The proof proceeds by defining $Y_k = \E[f(X) \mid X_1,\ldots,X_k]$ (using our notation) and claiming that this martingale has bounded differences. In the last line of the proof, it is claimed that ``Similarly it follows that \ldots". However, the desired inequality does not follow by the same argument, and can be false.

\medskip
\noindent{\bf Example.}
Consider the following random variables $X_1,\ldots,X_n$: $X_2,\ldots,X_n$ are independent and uniformly random in $\{0,1\}$. $X_1$ is the NAND function of $X_2,\ldots,X_n$, i.e.,~$X_1 = 1 - \prod_{i=2}^{n} X_i$. Note that $\Pr[X_1 = 1] = 1 - 1/2^{n-1}$.
We claim that $X_1,\ldots,X_n$ satisfy negative regression: 

Let $I,J \subset [n]$ be disjoint and $a \leq b \in \{0,1\}^J$. We consider the following cases:
\begin{compactitem}
\item If $1 \notin I \cup J$, then there is no dependence between $X_I$ and $X_J$.
\item If $1 \in I$, then the distribution of $X_I$ differs under $X_J=a,b$ only if $a \neq 1^J$ and $b = 1^J$. In this case, the $X_J = b$ is consistent with $X_i=0$ or $X_i=1$, while $X_J = a$ implies $X_1 = 1$. The other variables in $I$ are independent of $X_J$. Hence, conditioning on $X_J = a$ as opposed to $X_J=b$ can only increase the expectation of any monotone function $f(X_I)$. 
\item If $1 \in J$, then the distribution of $X_I$ differs under $X_J=a,b$ only if $a_1=0$ and $b_1=1$. In this case, $X_I = a$ implies that all the other variables are equal to $1$, while $X_I = b$ is consistent with any assignment to the other variables except all 1's. Therefore, conditioning on $X_J=a$ as opposed to $X_J=b$ can only increase the expectation of any monotone function $f(X_I)$. 
\end{compactitem}

\medskip

Now consider the martingale $Y_0, Y_1, \ldots, Y_n$, where $Y_k = \E[f(X) \mid X_1,\ldots,X_k]$ for $f(X) = \sum_{i=1}^{n} X_i$. 
$$Y_0 = \E[f(X)] = \sum_{i=1}^{n} \E[X_i] = (1 - \frac{1}{2^{n-1}}) + \frac12 (n-1) = \frac12 (n+1) - \frac{1}{2^{n-1}}.$$
If $X_1 = 0$ (which happens with probability $1/2^{n-1}$), this implies that all the remaining variables are equal to $1$, which means that $Y_1 = n-1$. 
Therefore, the difference between $Y_0$ and $Y_1$ can be $\Omega(n)$. Clearly, the issue here is the enormous influence of $X_1$ via its correlation with the remaining variables, and this is what motivates our adaptive ordering of variables.

\end{document}